\documentclass[11pt,hidelinks]{article}

\usepackage{orcidlink}

\title{Critical scaling profile for trees and connected subgraphs
\\
 on the complete graph}

\author{
Yucheng Liu\,\orcidlink{0000-0002-1917-8330}\thanks{Department of Mathematics,
	University of British Columbia,
	Vancouver, BC, Canada V6T 1Z2.
	Liu: \href{mailto:yliu135@math.ubc.ca}{yliu135@math.ubc.ca}.
	Slade: \href{mailto:slade@math.ubc.ca}{slade@math.ubc.ca}.
	}
\and
Gordon Slade\,\orcidlink{0000-0001-9389-9497}$^*$
}
\date{\vspace{-5ex}} 

\usepackage{amsmath, amssymb, amscd, amsthm, amsfonts}
\usepackage{graphicx} 
\usepackage{hyperref} 

\usepackage{booktabs}
\usepackage{xcolor}
\usepackage{ dsfont, bm}
\usepackage[title]{appendix}
\usepackage{comment}
\usepackage{enumerate}
\usepackage{enumitem}
\usepackage{cite}

\usepackage[textwidth=480pt,textheight=650pt,centering]{geometry} 

\theoremstyle{plain}
\newtheorem{theorem}{Theorem}[section]
\newtheorem{lemma}[theorem]{Lemma}
\newtheorem{conjecture}[theorem]{Conjecture}
\newtheorem{proposition}[theorem]{Proposition}

\numberwithin{equation}{section}

\newcommand{\eps}{\varepsilon}

\newcommand{\Z}{\mathbb{Z}}

\newcommand{\R}{\mathbb{R}}

\newcommand{\K}{\mathbb{K}}
\renewcommand{\P}{\mathbb{P}}
\newcommand{\T}{\mathbb{T}}

\newcommand{\del}{\partial}

\newcommand{\inv}{^{-1}}

\newcommand{\half}{\frac{1}{2}}

\newcommand{\nl}{\nonumber \\}

\DeclareMathOperator\arctanh{arctanh}

\providecommand{\abs}[1]{\lvert#1\rvert}

\newcommand{\supk}[1]{^{(#1)}}

\newcommand{\supperc}{^{\rm {perc}}}

\usepackage{mathrsfs}

\newcommand{\nnb}{\nonumber\\}

\newcommand{\const}{\mathrm{const}}

\begin{document}
\maketitle

\begin{abstract}
We analyse generating functions for trees and
for connected subgraphs on the complete graph,
and identify a single scaling profile which applies for
both generating functions
in a critical window.
Our motivation comes from the analysis of the finite-size
scaling of lattice trees
and lattice animals on a high-dimensional discrete torus, for which we
conjecture that the identical profile applies in dimensions $d \ge 8$.
\end{abstract}




\section{Main results}

\subsection{Results}

The enumeration of trees and connected graphs has a long history.
We are motivated by problems arising in the critical behaviour of
branched polymers
in equilibrium statistical mechanics, to consider certain generating
functions for the number of trees and connected subgraphs in the complete graph $\K_V$
on $V$ labelled vertices.  The vertices are labelled as
$\mathbb V = \{ 0, \dots, V-1 \}$ and the edge set is
$\mathbb E = \{ \{ x,y \}: x,y \in \mathbb V ,\ x\neq y\}$.
Our interest is in the asymptotic behaviour as $V \to \infty$.

We define \emph{one-point functions}
\begin{align}
\label{eq:G0def}
G_{V,0}^t(p) = \sum_{T \ni 0}   \Bigl(\frac{p}{eV} \Bigr)^{\abs T} ,
\qquad
G_{V,0}^a(p) = \sum_{A \ni 0}   \Bigl(\frac{p}{eV} \Bigr)^{\abs A} ,
\end{align}
where the first sum is over all labelled trees $T$ in $\K_V$ containing the vertex $0$,
the second sum is over all labelled connected subgraphs $A$ containing $0$,
and $|T|$ and $|A|$ denote the number of edges in $T$ and $A$.
The division of $p$ by $eV$ is a normalisation to make
$p=1$ correspond to a critical value.
We also study the
\emph{two-point functions}
\begin{align}
\label{eq:G01def}
G_{V,01}^t(p) = \sum_{T \ni 0,1}   \Bigl(\frac{p}{eV} \Big)^{\abs T} ,
\qquad
G_{V,01}^a(p) = \sum_{A \ni 0,1}   \Bigl(\frac{p}{eV} \Big)^{\abs A} ,
\end{align}
where the sums now run over trees or connected subgraphs containing the distinct
vertices $0,1$.
To avoid repetition, when a formula applies to both
trees and connected subgraphs we often omit the superscripts $t,a$.
With this convention, we define the \emph{susceptibility}
\begin{equation}
\label{eq:chidef}
    \chi_V(p) = G_{V,0}(p)+(V-1)G_{V,01}(p).
\end{equation}
We are particularly interested in values of $p$ in a \emph{critical window}
$p=1+sV^{-1/2}$ around the critical point, with $s \in \R$.

We define the profile
\begin{equation}
\label{eq:Iprofile}
    I(s)  =
    \frac{e}{\sqrt{2\pi}}\int_0^\infty e^{-\frac 12 x^2+sx}\frac{1}{\sqrt{x}}
    \mathrm d x
    \qquad (s \in \R).
\end{equation}
The profile can be rewritten in terms of
a \emph{Fax\'en integral} \cite[p.~332]{Olve97} as
$I (s) = e\pi^{-1/2}2^{-5/4} \mathrm{Fi} ( \frac{1}{2}, \frac{1}{4} ; \sqrt{2} s )$,
and its asymptotic behaviour is given by
\cite[Ex.~7.3, p.~84]{Olve97} to be
\begin{equation}
    I (s)
    \sim
    \begin{cases}
    e |2s|^{-1/2} & (s \to -\infty)
    \\
    e s^{-1/2}e^{s^2/2} & (s \to +\infty) ,
    \end{cases}
\end{equation}
where $f \sim g$ means $\lim f/g = 1$.
Our main result is the following theorem.

\begin{theorem} \label{thm:profile}
For both trees and connected subgraphs, and for all $s \in \R$,
as $V \to \infty$ we have
\begin{align}
\label{eq:G01-profile}
     G_{V,01}(1 + sV^{-1/2})
    & \sim V^{-3/4} I(s)    ,
\\
\label{eq:LTK-profile}
    \chi_V(1 + sV^{-1/2})
    &\sim
    V^{1/4} I(s)
    .
\end{align}
\end{theorem}

The proof of Theorem~\ref{thm:profile} uses a uniform bound on the
one-point function.
The following theorem gives a statement that is more precise than a bound.
It involves
the principal branch $W_0$ of the Lambert function \cite{CGHJK96}, which
solves $W_0e^{W_0} =z$ and has power series
\begin{equation}
\label{eq:Lambert}
    W_0(z) = \sum_{n=1}^\infty \frac{(-n)^{n-1}}{n!}z^n.
\end{equation}
The solution to $W_0e^{W_0} = -1/e$ is achieved by the
particular value $W_0(-1/e)=-1$.

\begin{theorem}
\label{thm:G0}
For both trees and connected subgraphs, for all $s \ge 0$, and for all sequences
$p_V$ with $p_V \le 1+sV^{-1/2}$ and $\lim_{V\to\infty}p_V=p\in [0,1]$,
\begin{align}
\label{eq:G0lim}
    \lim_{V\to\infty} G_{V,0}(p_V)
    &=
    \sum_{n=1}^\infty \frac{ n^{n-1} }{n!} \Big( \frac{p}{e} \Big)^{n-1}
    =
    -\frac{e}{p}W_0\Big(\!\!-\frac{p}{e} \Big)
    .
\end{align}
In particular, if $p=1$ then
$\lim_{V\to\infty} G_{V,0}(p_V) = e$.
\end{theorem}

\noindent
\textbf{Notation:} We write $f \lesssim g$ if there is a $C>0$ such that $f(x) \le C g(x)$
for all $x$ of interest.

\subsection{Method of proof}

To prove \eqref{eq:G01-profile}, it suffices to prove
\eqref{eq:LTK-profile} and \eqref{eq:G0lim}, since
when $p=1+sV^{-1/2}$, by definition of $\chi_V$
we then have
\begin{equation}
\label{eq:Gchi-ratio}
    G_{V,01} = \frac{ \chi_V - G_{V,0} } {V - 1 } \sim \frac {\chi_V} V .
\end{equation}

\subsubsection{Trees}

By Cayley's formula, the number of trees on $n$ labelled vertices is $n^{n-2}$.
By decomposing the sum defining $G_{V,0}^t(p)$ according to the number $n$ of vertices in the tree,
and by counting the number of ways to choose $n-1$ vertices other than $0$, we have
\begin{align} \label{eq:G0}
    G_{V,0}^t(p) &= \sum_{n=1}^{V}
	\binom{V-1}{n-1}
	n^{n-2}  \Big(\frac{p}{eV} \Big)^{n-1} .
\end{align}
Similarly,
by counting the number of ways to choose $n-2$ vertices other than $0$ and $1$, we have
\begin{equation} \label{eq:G01t}
    G_{V,01}^t(p) =
    \sum_{n=2}^{V}
	\binom{V-2}{n-2}
	n^{n-2}  \Big(\frac{p}{eV} \Big)^{n-1} .
\end{equation}
Since
\begin{equation}
\binom{V-1}{n-1} + (V-1) \binom{V-2}{n-2} = n\binom{V-1}{n-1},
\end{equation}
it follows from \eqref{eq:chidef} that the susceptibility is given by
\begin{align} \label{eq:chit}
\chi_{V}^t(p)
    &=
    \sum_{n=1}^{V}
	\binom{V-1}{n-1}
	n^{n-1}  \Big(\frac{p}{eV} \Big)^{n-1}.
\end{align}
For trees,
we prove Theorems~\ref{thm:profile}--\ref{thm:G0}
by directly analysing the above series for $\chi_V^t$ and $G_{V,0}^t$.
The profile $I(s)$ for $\chi_V^t(1 + sV^{-1/2})$ arises from a Riemann sum limit.

\subsubsection{Connected subgraphs}

For connected subgraphs,
we will show that the contribution to $\chi_V^a, G_{V,01}^a$
from connected subgraphs with cycles
is much smaller than the contribution from trees.
Let $C(n,n-1+\ell)$ denote the number
of connected graphs on $n$ labelled vertices with exactly $n-1+\ell$ edges,
i.e., with $\ell$ surplus edges.
The surplus must be zero for $n=1,2$.
For $n\ge 3$, we define the \emph{surplus generating function}
\begin{equation} \label{eq:def_S}
    S(n,z)  =
    \sum_{\ell= 1}^\infty C(n,n-1+\ell)
	z^{\ell} .
\end{equation}
Note that terms in the above sum are zero unless
$\ell \le \binom{n} 2 - (n-1)$,
and that the tree term ($\ell=0$) is absent.

We decompose the sums defining $G_{V,0}^a$ and $G_{V,01}^a$ according to the number $n$ of vertices in the connected subgraph, and we further distinguish whether or not the subgraph contains surplus edges.
This leads to the decomposition 
\begin{align}
    G_{V,0}^a(p) &=
    G_{V,0}^t(p)
    +  \Delta_{V,0}(p),
\\
    \chi_{V}^a(p) &=
    \chi_{V}^t(p)
    +
    \Delta_{V}(p),
\end{align}
with
\begin{align}
\label{eq:Delta0}
    \Delta_{V,0}(p) &=
    \sum_{n=3}^{V}
	\binom{V-1}{n-1}
	S\Big(n,\frac p {eV}\Big)
	\Big(\frac{p}{eV} \Big)^{n-1} ,
\\
\label{eq:chiDelta}
    \Delta_{V}(p) & =
    \sum_{n=3}^{V}
	\binom{V-1}{n-1} n
	S\Big(n,\frac p {eV}\Big)
	\Big(\frac{p}{eV} \Big)^{n-1}.
\end{align}
Given Theorems~\ref{thm:profile}--\ref{thm:G0} for trees,
we prove Theorems~\ref{thm:profile}--\ref{thm:G0} for connected subgraphs by showing that, for all $s\in \R$,
\begin{align}
\label{eq:Delta0bd}
    \lim_{V\to\infty} \Delta_{V,0}(1+sV^{-1/2}) & =0,
\\
\label{eq:Deltabd}
    \lim_{V\to\infty} V^{-1/4}\Delta_V(1 + sV^{-1/2})
    &=
    0.
\end{align}
The proof is more subtle than for trees
and requires estimates on the surplus generating function.
As we discuss later, a precise but cumbrous asymptotic formula for $C(n,n+k)$ is
given in \cite[Corollary~1]{BCM90}.  We use that formula to prove the following
useful explicit bound.
By convention, $k^k = 1$ when $k=0$.

\begin{proposition} \label{prop:Cnk-intro}
Let $n \ge 3$ and $N = \binom n 2$.
For $0 \le k  \le n$, we have
\begin{align} \label{eq:Cnk-intro}
C(n, n+k) \lesssim   \binom N {n+k}
	\bigg( \frac 2 e \bigg)^n   \bigg( \frac{ e n }{ k }\bigg)^{k/2} .
\end{align}
\end{proposition}

Proposition~\ref{prop:Cnk-intro} is most useful
when the surplus $\ell = k+1$ is small but of order $n$.
This is a delicate region when controlling the surplus generating function,
and the precise constant $e$ in the last factor of \eqref{eq:Cnk-intro} is important.
For a larger surplus, we simply bound $C(n,n+k)$
by the total number of graphs (connected or not) on $n$ vertices with
$n+k$ edges, which is $\binom{N}{n+k}$.
Together, these bounds provide enough control on $S(n,p/(eV))$ to
prove \eqref{eq:Delta0bd}--\eqref{eq:Deltabd}.

\subsection{Motivation}

Theorem~\ref{thm:profile} is motivated by a broader emerging theory of
finite-size scaling in statistical mechanical models above their upper
critical dimensions.  The theory involves a family of profiles expressed in terms of
the functions
\begin{equation}
    I_k(s) = \int_0^\infty x^{k} e^{-\frac 14 x^4 - \frac 12 s x^2} \mathrm dx
    \qquad
    (s\in \R, \; k>-1)   .
\end{equation}
A change of variables
transforms the profile $I$ of \eqref{eq:Iprofile} into
$I(s)= e 2^{1/4}\pi^{-1/2} I_0(-\sqrt{2}s)$.
The general theory is described in \cite{LPS25-universal} with
references to the extensive physics and mathematics literature.

Given an integer $d \ge 2$, infinite-volume models can be formulated on
a transitive graph $\mathbb{G}=(\Z^d,\mathbb{E})$, whose
edge set $\mathbb{E}$ has a finite number of edges containing the origin
and is invariant under the symmetries of $\Z^d$.
Above an \emph{upper critical dimension}
$d_{\rm c}$, for many models it has been proven that
the critical exponents that describe the critical behaviour are the same
as the corresponding exponents when the model is formulated on a
regular tree or on the complete graph.  The tree and complete graph settings
are easy to analyse.
Finite-volume models (with periodic boundary conditions)
can instead be formulated on a discrete torus $\mathbb{G}_r = (\T_r^d,\mathbb{E}_r)$
of period~$r$.  At and above the upper critical dimension, the torus models are known
or conjectured to have critical behaviour analogous to that seen on the complete
graph, with an interesting ``plateau'' phenomenon involving a universal profile
which is often expressed in terms of $I_k$.
The value of $k$ depends on the model.
Dimensions $d<d_{\rm c}$ are conjectured to exhibit different scaling, with no
plateau or profile.

\smallskip\noindent
{\bf Lattice trees and lattice animals:}
A \emph{lattice animal} is
a finite connected subgraph of $\mathbb{G}$, and a \emph{lattice tree} is
an acyclic lattice animal.
The critical behaviour of lattice trees and lattice animals is at least as difficult
as is the case for the notoriously difficult self-avoiding walk.
Despite significant interest
from chemists and physicists for over half a century,
due to applications to branched polymers \cite{Jans15},
the critical behaviour is understood mathematically only
in dimensions $d>d_{\rm c} = 8$.
For $d>8$, it has been proved using the lace expansion that for sufficiently
large edge sets $\mathbb{E}$ (or for nearest-neighbour edges with $d$ sufficiently
large), lattice trees and lattice animals at the critical point
both have the same behaviour as a critical branching process
\cite{CFHP23,DS98,HS90b,HS92c}.

For $x\in \Z^d$, let $c_m(x)$ denote the number of lattice trees or lattice animals containing
$0,x$ and having exactly $m$ bonds.  The one-point functions,
two-point functions, and
susceptibilities are defined by
\begin{equation}
    g(z) = \sum_{m=0}^\infty c_m(0)z^m,
    \quad
    G_z(x) = \sum_{m=0}^\infty c_m(x)z^m,
    \quad
    \chi(z) = \sum_{x\in \Z^d} G_z(x).
\end{equation}
The radius of convergence $z_{\rm c}$ (the \emph{critical point})
of these series is finite and positive, and is strictly
smaller for animals than for trees \cite{GPSW94}.
High-dimensional versions and extensions of Theorem~\ref{thm:G0} for $g(z_c)$ are
proved in \cite{KS24,MS13}.
The analogous quantities for trees and animals on
the torus $\T_r^d$ are denoted $g_r(z)$, $G_{r,z}(x)$, $\chi_r(z)$.
These are polynomials in $z$, so they define entire functions of $z$.
Nevertheless, for large $r$ the infinite-volume critical point $z_c$ plays
a role in the scaling.  We denote the volume of the torus by $V=r^d$.

Our computation of the profile $I$ for the two-point
function and susceptibility in Theorem~\ref{thm:profile}
supports the following conjecture from \cite{LS25a} that the profile $I_0$
(just a rescaled $I$) occurs
for both lattice trees and lattice animals on the torus,
above the upper critical dimension.

\begin{conjecture}
For lattice trees and lattice animals on $\T_r^d$ with $d>8$,
there are constants $a_d<0$ and $b_d>0$ (different constants for trees
and animals) such that, as $V=r^d \to \infty$,
\begin{equation}\begin{aligned}
\label{eq:LTLAplateau}
    G_{r,z_c+sV^{-1/2}}(x) -G_{z_c}(x)
    &\sim b_d V^{-3/4} I_0(a_d s),
    \\
    \chi_r(z_c+sV^{-1/2}) &\sim b_d V^{1/4} I_0(a_d s).
\end{aligned}\end{equation}
\end{conjecture}

In \eqref{eq:LTLAplateau}, the torus point $x$ is identified with its
representative in $\Z^d \cap (-\frac r2,\frac r2]^d$ in the evaluation of $G_{z_c}(x)$.
For $d>8$, $G_{z_c}(x)$ decays as $|x|^{-(d-2)}$ \cite{Hara08,HHS03}, and the constant term
of order $V^{-3/4}= r^{-3d/4}$ dominates the Gaussian decay over most of the torus.
This is the ``plateau'' phenomenon.
On the complete graph,
the decaying term $\abs x^{-(d-2)}$ is absent, and
only the constant term occurs for $G_{V,01}$,
as in \eqref{eq:G01-profile}.
For $d=d_{\rm c}=8$, the conjecture is modified to include logarithmic corrections
to the window scale $V^{-1/2}$, the plateau scale $V^{-3/4}$,
and the susceptibility scale $V^{1/4}$, but with the identical
profile $I_0$.

\smallskip\noindent
{\bf Self-avoiding walk:}
Self-avoiding walk on the complete graph $\mathbb{K}_V$ is exactly solvable \cite{Slad20}.
For $1 \le n \le V-1$,
let $c_{V,n}(0,1) = \prod_{j=2}^n (V-j)$ denote the number of $n$-step self-avoiding
walks from $0$ to $1$ on $\mathbb{K}_V$.
Let $S_{V,01}(p) = \sum_{n=1}^{V-1} c_{V,n}(0,1)(p/V)^n$ and let $\chi_V^{\rm SAW}(p)
= 1+(V-1)S_{V,01}(p)$.  It is proved in \cite{Slad20}
(see also \cite[Appendix~B]{MPS23}) that,
as $V \to \infty$,
\begin{equation}\begin{aligned}
     S_{V,01}(1 + sV^{-1/2})
    & \sim  (2V)^{-1/2} I_1(-\sqrt{2} s)
    ,
\\
    \chi_V^{\rm SAW}(1 + sV^{-1/2})
    &\sim
    2^{-1/2}V^{1/2} I_1(-\sqrt{2}s)
    .
\end{aligned}\end{equation}
In \cite{MPS23,PS25}, the same profile $I_1$ is conjectured
to apply to the self-avoiding walk on $\T_r^d$ for $d \ge 4$,
in the sense that the two-point function and susceptibility obey
the analogue of \eqref{eq:LTLAplateau} with the right-hand sides replaced
respectively by
$b_d V^{-1/2} I_1(a_d s)$ and $ b_d V^{1/2} I_1(a_d s)$.
The conjectured log corrections for $d=4$ are indicated in
\cite[Section~1.6.3]{MPS23}.

\smallskip\noindent
{\bf Spin systems:}
The plateau for spin systems in dimensions $d \ge d_{\rm c}=4$
is discussed in \cite{LPS25-Ising,LPS25-universal,PS25}, including rigorous results for
a hierarchical $|\varphi|^4$ model and conjectures for spin systems on the torus.
The relevant profile for $n$-component spin systems is
\begin{equation}
    f_n(s) =
    \frac{\int_{\R^n} |x|^2 e^{-\frac 14 |x|^4 - \frac s2 |x|^2} dx}
    {n\int_{\R^n} e^{-\frac 14 |x|^4 - \frac s2 |x|^2} dx}
    =
    \frac{ I_{n+1}  (s)}{n I_{n-1}  (s)}.
\end{equation}
The profile $f_1$ has been proven to occur for the Ising model on the
complete graph (Curie--Weiss model); a recent reference is \cite{BBE24}.
As $n \to 0$, the profile $f_n(s)$ converges to $I_1(s)$, which is consistent
with the conventional wisdom that the spin model with $n=0$ corresponds to the self-avoiding walk.

\smallskip\noindent
{\bf Percolation:}
Percolation has been extensively studied both on infinite lattices
\cite{Grim99} and
on the complete graph (the Erd\H{o}s--R\'enyi random graph)
\cite{JLR00}.
This is a probabilistic model in which the
cluster containing $0$ is a connected subgraph $A \ni 0$
with weight
$p^{\abs A} (1-p)^{\abs{\del A}}$,
where $\abs A$ denotes the number of edges in $A$, and
$\del A$ denotes the
set of edges which are not in $A$ but are incident to one or two vertices in $A$.
On the complete graph,
we divide $p$ by $V$ (not by $eV$ as in \eqref{eq:G01def}) to make the critical value $p=1$.
Thus we define the \emph{two-point function}
\begin{equation}
\tau_{V,01}(p)
= \P_{p/V}(0 \leftrightarrow 1)
= \sum_{A \ni 0,1}   \Bigl(\frac{p}{V} \Big)^{\abs A}
	\Bigl(1-\frac{p}{V} \Big)^{\abs {\del A}}
\end{equation}
and the \emph{susceptibility} (expected cluster size) $\chi_V \supperc(p) = 1 +(V-1)\tau_{V,01}(p)$.
Our conjecture for an analogue of Theorem~\ref{thm:profile}
for percolation on the complete graph is
as follows.
It involves the Brownian excursion
$W^*$ of length $1$, and the
moment generating function
$\Psi(x) = \mathbb{E} \exp[ x\int_0^1 W^*(t)\mathrm dt] $ for the Brownian excursion area.

\begin{conjecture}
\label{conj:perc}
For $s\in \R$, let
\begin{equation}
\label{eq:fperc}
f_{\rm perc}(s) = \int_0^\infty x^2 \mathrm d\sigma_s, \qquad
\mathrm d \sigma_s =
	\frac 1 { \sqrt{2\pi} } x^{-5/2} \Psi(x^{3/2})
    e^{-\frac 1 6 x^3 + \frac s 2 x^2 - \frac{ s^2} 2 x}
    \mathrm dx
    .
\end{equation}
Then, for some $a,b>0$, as $V\to\infty$ we have
\begin{equation} \label{eq:perc_conj}
\begin{aligned}
    \tau_{V,01}(1+sV^{-1/3}) & \sim b V^{-2/3} f_{\rm perc}(as),
    \\
    \chi_{V} \supperc(1+sV^{-1/3}) & \sim b V^{1/3} f_{\rm perc}(as).
\end{aligned}
\end{equation}
\end{conjecture}

Note the different powers of $V$ in \eqref{eq:perc_conj} compared to
\eqref{eq:LTLAplateau} and \eqref{eq:G01-profile}--\eqref{eq:LTK-profile}.
The powers of $V$ in \eqref{eq:perc_conj} are well-known, but to our knowledge
the occurrence of the profile has not been proved.
On the torus $\T_r^d$ with $d>6$, the powers
$V^{-1/3}, V^{-2/3}, V^{1/3}$
are proved in \cite{HMS23},
and the role of $f_{\rm perc}$ was first conjectured in \cite[Appendix~C]{LPS25-Ising}.

The origin of the conjecture is as follows.
The properly rescaled cluster size (without expectation) is known to converge in distribution to a random variable described by the Brownian excursion \cite{Aldo97},
and the limiting random variable is characterised by a point process \cite{JS07}.
The measure $\sigma_s$ is the intensity of the point process and is found in \cite[Theorem~4.1]{JS07}.
The point process describes cluster sizes, in the sense that
\begin{equation}
n^{-2k/3} \sum_i \abs{ C_i }^k
\Rightarrow
\int_0^\infty x^k \mathrm d\sigma_s
\qquad
(k \ge 2)
\end{equation}
in distribution \cite{Aldo97}.
The $k=2$ case corresponds to $\chi_V\supperc$ and identifies $f_{\rm perc}(s)$.

\section{Proof for trees}

We begin with an elementary lemma.

\begin{lemma} \label{lem:b}
Let $\gamma \ge 0$, $\kappa > 0$, and $\lambda \in \R$.
There is a constant $C_{\kappa, \lambda}>0$ such that
\begin{align} \label{eq:b_bound}
\sum_{n=\lceil b \sqrt V \rceil}^{V}
	\frac{ 1 }{n^\gamma}
	e^{- \kappa n^2/V} e^{\lambda n/\sqrt{V}}
\le C_{\kappa,\lambda} b^{-\gamma} V^{ (1 - \gamma) /2} .
\end{align}
for all $V$ and for all $b$ sufficiently large (depending on $\kappa,\lambda$).
\end{lemma}

\begin{proof}
Since $n \ge b\sqrt V$ and $\gamma \ge 0$,
the left-hand side of \eqref{eq:b_bound} is bounded by
\begin{align}
\frac{1}{b^\gamma V^{\gamma / 2}}
	\sum_{n=\lceil b \sqrt V \rceil}^{\infty}
	e^{- \kappa n^2/V} e^{\lambda n/\sqrt{V}} .
\end{align}
For $b$ sufficiently large (depending on $\kappa,\lambda$),
the summand above is monotone decreasing in $n$,
so we can bound the sum by the integral
\begin{align}
\int_{b \sqrt V - 1}^\infty
e^{- \kappa ( y / \sqrt V )^2} e^{\lambda( y/\sqrt{V}) }
\mathrm dy
\le C_{\kappa, \lambda} \sqrt V ,
\end{align}
and the desired result follows.
\end{proof}

\begin{proof}[Proof of Theorem~\ref{thm:profile} for trees]
We use \eqref{eq:chit} and drop the superscript $t$.
Fix $s\in \R$.
For $p = 1 + s V^{-1/2}$, by
combining $V^{-(n-1)}$ with the binomial coefficient,
we have
\begin{equation}
\label{eq:chit-sum}
    \chi_V(1+sV^{-1/2})
    = \sum_{n=1}^{V}
    \Bigl(\prod_{j=1}^{n-1}( 1 - \frac jV )\Bigr) \frac 1 {(n-1)!}
   \frac{n^{n-1}}{e^{n-1}} \Big(1+\frac s {\sqrt{V}} \Big)^{n-1}.
\end{equation}
Let $0 < a < 1 < b < \infty$.
We divide the sum over $n$  into three parts
$\chi_V \supk 1, \chi_V \supk 2, \chi_V \supk 3$, which respectively
sum over $n$ in the intervals $[1, a \sqrt V)$, $[ a \sqrt V, b \sqrt V]$, $(b \sqrt V , V]$.
We will prove that
\begin{equation}
\label{eq:claimLTK13}
\chi_V \supk 1 \lesssim e^{a\abs s} ( 1 + a^{1/2} V^{1/4}),
\qquad
\chi_V \supk 3 \le C_{\abs s} b^{-1/2} V^{1/4}
\end{equation}
for all $a>0$ and all $b$ sufficiently large,
and that
\begin{equation}
\label{eq:claimLTK2}
\lim_{V\to \infty} V^{-1/4} \chi_V^{(2)} = \int_a^b f(x) \mathrm d x ,
\qquad
f(x) = \frac{e}{\sqrt{2\pi}}e^{-x^2/2}\frac{1}{\sqrt{x}} e^{sx}
\end{equation}
for all $a,b$.
These claims imply that
\begin{equation}
\int_a^b f(x) \mathrm d x
\le \liminf_{V\to \infty} \frac{ \chi_V }{V^{1/4}}
\le \limsup_{V\to \infty} \frac{ \chi_V }{V^{1/4}}
\le  C e^{a\abs s}a^{1/2}+ \int_a^b f(x) \mathrm d x + C_{|s|}  b^{-1/2}
\end{equation}
for all $a > 0$ and all $b$ sufficiently large.
Since $\chi_V$ does not depend on $a$ or $b$,
by taking the limits $a\to 0$, $b\to \infty$, we obtain
$\lim_{V\to \infty} V^{-1/4} \chi_V = \int_0^\infty f$,
which is the desired result \eqref{eq:LTK-profile}.

It remains to prove the claims \eqref{eq:claimLTK13}--\eqref{eq:claimLTK2}.
Let
\begin{equation}
    b_{n} = \frac{n^{n-1}}{(n-1)!e^{n-1}} ,
\end{equation}
which obeys
$b_{n}  \lesssim 1/\sqrt{n}$,
by Stirling's formula.
Using this in the sum for $\chi_V \supk 1$,
and using $1 + s/\sqrt V \le e^{ \abs s / \sqrt V}$, we get
\begin{equation} \label{eq:pf_tree_1}
\chi_V^{(1)}
\lesssim \sum_{n=1}^{\lfloor a \sqrt V \rfloor}  1 \frac{1}{\sqrt {n} } e^{ |s| n / \sqrt V}
\le e^{a\abs s} \sum_{n=1}^{\lfloor a \sqrt V \rfloor}  \frac{1}{\sqrt {n} }
\lesssim e^{a\abs s}  ( 1 +  a^{1/2} V^{1/4}  ) ,
\end{equation}
as claimed.
For $\chi_V \supk 3$,
we also need a bound on the product over $j$.
Using $1-x \le e^{-x}$,
we have
\begin{equation}
\prod_{j=1}^{n-1}\big( 1 - \frac jV \big)
\le \exp\Bigl\{ - \frac 1 V \sum_{j=1}^{n-1}  j  \Bigr\}
 = \exp \Bigl\{ - \frac { n (n-1)} {2V} \Bigr\}    .
\end{equation}
By Lemma~\ref{lem:b} with $\gamma = \kappa = \half$ and $\lambda=|s|$,
this implies that, for all $b$ sufficiently large,
\begin{equation} \label{eq:pf_tree_3}
    \chi_V^{(3)}
\lesssim
    \sum_{n= \lceil b \sqrt V \rceil}^{V}
    e^{-n^2/2V}
     e^{n/2V}
    \frac{1}{\sqrt{ n}}
    e^{|s|n/\sqrt{V}}
\lesssim e^{1/2} b^{-1/2} V^{1/4}.
\end{equation}

Finally, for $\chi_V \supk 2$ we fix $a,b$ and use the asymptotic formulas
\begin{align}
\Big( 1 + \frac s {\sqrt V} \Big)^{n-1}
&= \exp \Big\{ (n-1) \log( 1 + \frac s {\sqrt V} ) \Big\}
= e^{s n/ \sqrt V} \Big[ 1   + O\Big(\frac 1 { \sqrt V }\Big)
	+ O\Big(\frac n V\Big) \Big] ,
\\
\prod_{j=1}^{n-1}\big( 1 - \frac jV \big)
&= \exp\Big\{\sum_{j=1}^{n-1} \log(1-\frac j V)\Big\}
= e^{-n^2/2V} \Big[ 1 +O\Big(\frac n V\Big) + O\Big(\frac {n^3} {V^2}\Big)\Big] ,
\end{align}
which follow from Taylor expansion of the logarithm 
(the constants here depend on $s$).
Since $n \in [a\sqrt V, b\sqrt V]$,
the above, together with the fact that $b_{n} =
\frac{ e }{\sqrt{2\pi n}}[1+O(1/n)]$
by Stirling's formula, give
\begin{align}
    \chi_V^{(2)} &  =
    \sum_{n= \lceil a \sqrt V \rceil}^{ \lfloor b \sqrt V \rfloor}
    e^{-n^2 / 2V}
    \frac{e}{\sqrt{2\pi n}}
    e^{sn/\sqrt{V}}
    \Big[ 1 + O\Big(\frac 1 {\sqrt V}\Big) \Big]
    .
\end{align}
The desired limit then follows from the observations that
the leading term of $V^{-1/4} \chi_V \supk 2$
is a Riemann sum for the integral $\int_a^b f$ with mesh size $V^{-1/2}$.
\end{proof}

\begin{proof}[Proof of Theorem~\ref{thm:G0} for trees]
We use \eqref{eq:G0} and again drop the superscript $t$.
Fix $s\ge 0$. Let $p_V$ be a sequence with $p_V \le 1 + s V^{-1/2}$ and $p_V \to p \in [0,1]$.
Similarly to \eqref{eq:chit-sum} and with an additional
factor of $n$ in the denominator,
\begin{equation}
\label{eq:G0t-sum}
    G_{V,0}(p_V)
    = \sum_{n=1}^{V}
    \Big(\prod_{j=1}^{n-1}( 1 - \frac jV )\Big) \frac 1 {n!}
   \frac{n^{n-1}}{e^{n-1}} p_V^{n-1}.
\end{equation}
Let $N, b\ge 1$.
We divide the sum over $n$  into three parts
$G_V \supk 1, G_V \supk 2, G_V \supk 3$, which respectively
sum over $n$ in the intervals $[1, N]$, $( N, b \sqrt V]$, $(b \sqrt V , V]$.
For a fixed $N$, we immediately get
\begin{align}
\label{eq:G0N}
    \lim_{V\to \infty} G^{(1)}_{V}(p_V) =
    \sum_{n=1}^N \frac{1}{n!} \frac{n^{n-1}}{e^{n-1}}p^{n-1},
\end{align}
which dominates the sum.
Indeed,
using monotonicity of the generating function,
for $G_V \supk 2$ we can proceed as in \eqref{eq:pf_tree_1} to bound
\begin{align}
G_V \supk 2 (p_V)
\le G_V \supk 2 (1 + s V^{-1/2})
\le e^{bs} \sum_{n=N}^{ \lfloor b\sqrt V \rfloor} \frac 1 {n^{3/2}}
\lesssim \frac{ e^{bs} } { \sqrt N} .
\end{align}
For $G_V \supk 3$, we can argue as in \eqref{eq:pf_tree_3}
but with an additional factor $n$ in the denominator,
and use Lemma~\ref{lem:b} with $\gamma = \frac 3 2$
to get $G_V \supk 3 (p_V) \lesssim b^{-3/2}V^{-1/4}$
for $b$ sufficiently large.
Together, we obtain
\begin{align}
\sum_{n=1}^N \frac{1}{n!} \frac{n^{n-1}}{e^{n-1}}p^{n-1}
\le \liminf_{V\to \infty} G_{V,0}
\le \limsup_{V\to \infty} G_{V,0}
\le \sum_{n=1}^N \frac{1}{n!} \frac{n^{n-1}}{e^{n-1}}p^{n-1}
	+  \frac{ C e^{bs} } { \sqrt N}
\end{align}
for all $N \ge 1$ and all $b$ sufficiently large.
Since $G_{V,0}$ does not depend on $N$,
we can take the limit $N \to \infty$ to conclude the desired result \eqref{eq:G0lim}.
\end{proof}

\section{Proof for connected subgraphs}

\subsection{Bound on $C(n,n+k)$}

We use the asymptotic formula for $C(n,n+k)$ proved in \cite{BCM90}. We follow the notation in \cite{BCM90} and write
\begin{equation}
x = 1 + \frac k n .
\end{equation}
For $x > 1$, we define the function $y = y(x) \in (0,1)$ implicitly by
\begin{equation}
\label{eq:Taylor_xy}
x  = \frac{1}{2y} \log \biggl( \frac { 1+ y } { 1 - y } \biggr)
= \frac{1}{y}\arctanh y
= \sum_{m = 0}^\infty \frac{ y^{2m} } { 2m+1 }
 ,
\end{equation}
and we define the functions $\varphi(x)$, $a(x)$ by
\begin{gather} \label{eq:phi}
e^{\varphi(x)} = \frac{ 2 e^{-x} y^{1-x} }{ \sqrt{ 1 - y^2} } , \\
a(x) = x(x+1)(1-y) + \log( 1 - x + xy) - \half \log( 1 - x + x y^2 ) .
\end{gather}
Both $\varphi$ and $a$ extend continuously to $x=1$ by defining $y^{1-x} = 1$ at $x=1$ and defining $a(1) = 2 + \half \log\frac 3 2$.

Let $N = \binom n 2$.
It is proved in \cite[Corollary~1]{BCM90} that there are constants $w_k = 1 + O(1/k)$
for which
\begin{equation} \label{eq:BCM}
C(n,n+k) = w_k \binom N {n+k} e^{n \varphi(x)} e^{a(x)}
	\bigg[ 1 + O\bigg( \frac { (k+1)^{1/16} }{ n^{9/50} } \bigg) \bigg]
\end{equation}
uniformly in $0 \le k \le N - n$.
The constants $w_k$ are related to Wright's constants for the
the asymptotics of $C(n,n+k)$ with $k$ fixed \cite{Wrig77},
and they are related to the Brownian excursion area \cite{Spen97}.
We will simply bound $w_k$ by a constant.
The next lemma gives estimates for $\varphi(x)$ and $a(x)$.

\begin{lemma} \label{lem:aphi}
Let $x\ge 1$.
\begin{enumerate}
\item [(i)]
The function $a(x)$ is bounded.

\item [(ii)]
Let $t = \sqrt{3e}$ and $y = y(x)$. Then
\begin{equation} \label{eq:phi_bdd}
e^{\varphi(x) } \le \frac 2 e  \exp\Bigl\{  - \frac 1 3 y^2 \log \frac y t \Bigr\} ,
\end{equation}
and the right-hand side is monotonically increasing for $0 < y \le t/\sqrt e$.
\end{enumerate}
\end{lemma}

By considering the limit $x\to \infty$ ($y\to 1$),
we expect that the inequality \eqref{eq:phi_bdd}
becomes optimal with $t = (e/2)^3 \approx 2.51$, but we do not pursue this.
The weaker version with $t= \sqrt{3e}$ is sufficient for our purposes,
but to show the role of $t$ we keep it in our formulas.

\begin{proof}
(i) The function $a(x)$ is continuous on $[1,\infty)$ by definition,
and it satisfies $\abs{ a(x) } \lesssim x^2 (1-y) \sim 2x^2 e^{-2x}$ as $x\to \infty$ by \cite[Lemma~3.2]{BCM90}, so it is bounded.

\smallskip \noindent
(ii)
By the definitions of $\varphi(x)$ and $x$, and by the Taylor series for
$\log (1-y^2)$,
\begin{align}
\varphi(x) - \log \frac 2 e
&=  (1-x) ( 1 + \log y ) - \half \log ( 1 - y^2 )  \nl
&= -  \sum_{m=1}^\infty \frac{ y^{2m} } { 2m+1 } ( 1 + \log y )
	+  \sum_{m=1}^\infty \frac{ y^{2m} } { 2m }		\nl
&= - \frac 1 3 y^2 \log y + \frac 1 6 y^2 	+ \sum_{m=2}^\infty
	\frac{ y^{2m} }{2m+1} \Bigl(-\log y + \frac 1 {2m} \Bigr) .
\end{align}
We bound the series in the last line by a quadratic function, term by term.
For any $m\ge 2$, by calculus,
\begin{equation}
\max_{0 \le y \le 1} \bigl[ y^{2m-2} (-2m \log y + 1) \bigr]
= \frac{2m}{2m-2} e^{-1/m} .
\end{equation}
Then, with
$K = \max_{m\ge 2} \{ \frac{2m}{2m-2} e^{-1/m} \} = 2 e^{-1/2}$,
by \cite[0.234.8]{GR07} we have
\begin{equation}
\sum_{m=2}^\infty
	\frac{ y^{2m} }{2m+1} \Bigl(-\log y + \frac 1 {2m} \Bigr)
\le \sum_{m=2}^\infty \frac{ K y^2 }{(2m+1)(2m)}
= (1 - \log 2 - \frac 1 6) K y^2 .
\end{equation}
Therefore,
\begin{equation}
\varphi(x) - \log \frac 2 e
\le - \frac 1 3 y^2 \log y
	+ \biggl[ \frac 1 6 + (1 - \log 2 - \frac 1 6)K \biggr] y^2 .
\end{equation}
This implies \eqref{eq:phi_bdd} with any $t$ that obeys
$ \frac 1 3 \log t  \ge \frac 1 6 + (1 - \log 2 - \frac 1 6)K \approx 0.3367$.
In particular, we can take any $t \ge 2.75$, including $t = \sqrt{3e} \approx 2.85$.
Monotonicity of the upper bound in $0 < y \le t/\sqrt e$ is another calculus exercise.
\end{proof}

We now restate and prove Proposition~\ref{prop:Cnk-intro}.

\begin{proposition} \label{prop:Cnk}
Let $n \ge 3$, $N = \binom n 2$, and $t=\sqrt{3e}$.
For $0 \le \frac k n \le \frac {t^2}{3e}$, we have
\begin{align} \label{eq:Cnk}
C(n, n+k) \lesssim   \binom N {n+k}
	\bigg( \frac 2 e \bigg)^n   \bigg( \frac{ t^2 n }{ 3 k }\bigg)^{k/2} .
\end{align}
\end{proposition}

\begin{proof}
We use the asymptotic formula \eqref{eq:BCM},
and use that $w_k = 1 + O(1/k)$ is bounded.
The error term in \eqref{eq:BCM} is bounded by a constant since 
$k$ is at most linear in $n$.
The factor $e^{a(x)}$ is also bounded by a constant, by Lemma~\ref{lem:aphi}(i).
We therefore only need to estimate $e^{n\varphi(x)}$.
Since $0 \le \frac k n \le \frac {t^2}{3e}$ and $x = 1 + \frac k n$, by \eqref{eq:Taylor_xy} we have
$y(x) \le \sqrt{3(x-1)} = \sqrt{ 3k / n } \le t / \sqrt e$,
so Lemma~\ref{lem:aphi}(ii) gives
\begin{align}
e^{\varphi(x)} \le  \frac 2 e  \exp\Bigl\{  - \frac 1 3 y^2 \log \frac y t \Bigr\}
\le \frac 2 e  \exp\Bigl\{  - \frac {x-1}2 \log \frac {3(x-1)} {t^2} \Bigr\}
= \frac 2 e \bigg( \frac {t^2 n}{3k} \bigg)^{k/2n} .
\end{align}
The desired result then follows
by inserting the above into \eqref{eq:BCM}.
\end{proof}

For larger $\frac kn$ we simply use the fact that $C(n,n+k)$ is less than the total number of graphs (connected or not) on $n$ vertices with
$n+k$ edges, which is $\binom{N}{n+k}$.
For all $n \ge 2$ and $k \ge -1$, we have
\begin{equation} \label{eq:C_crude}
    C(n,n+k) \le \binom{N}{n+k} \le \frac{N^{n+k}}{(n+k)!}.
\end{equation}

\subsection{Bound on the surplus generating function}

We now prove useful bounds on the surplus generating function defined in \eqref{eq:def_S}:
\begin{align}
    S(n,z) & =
    \sum_{\ell= 1}^\infty C(n,n-1+\ell) z^{\ell}
     =
    \sum_{k=0}^\infty C(n,n+k) z^{k+1}.
\end{align}
The terms in the series are zero unless $k \le \binom{n} 2 - n$.
The goal is to prove that $S(n,z)$ is small relative to the number of trees $C(n,n-1) = n^{n-2}$.
We do this by decomposing the series into two parts corresponding to sparse and dense graphs.
We define
\begin{equation}
    A(n,z)  = \frac 1 { n^{n-2} }
    \sum_{k=0}^{n}
    C(n,n+k) z^{k+1} ,
    \quad
    B(n,z)  = \frac 1 { n^{n-2} }
    \sum_{k= \lfloor \frac 12 n \rfloor}^\infty
    C(n,n+k) z^{k+1} ,
\end{equation}
so that
\begin{equation} \label{eq:SAB}
S(n,z) \le n^{n-2} \big( A(n,z) + B(n,z) \big) .
\end{equation}

\begin{lemma}[Sparse connected graphs]
\label{lemma:sparse}
Let $n \ge 3$, $z \ge 0$, and $t=\sqrt{3e}$.
\begin{enumerate}
\item[(i)]
If $n^{3/2} z \le b$, then $A(n,z) \le C_b n^{3/2} z$ for some $C_b > 0$.

\item[(ii)]
If $\eps > 0$, then
\begin{equation}
    A(n,z) \le C_\eps
    \exp\Bigl\{ \big(\frac{1}{24}+\eps \big)  e t^2 z^2 n^3  \Bigr\}
\end{equation}
for some $C_\eps > 0$.
\end{enumerate}
\end{lemma}

\begin{proof}
Since $\frac {t^2}{3e} = 1$,
we can apply Proposition~\ref{prop:Cnk} to estimate $C(n,n+k)$.
For the binomial coefficient in \eqref{eq:Cnk},
we use Stirling's formula, $n+k \ge n$, and $N = \binom n 2 =  \half n(n-1)$ to
see that
\begin{align}
\binom N {n+k}
\le \frac{ N^{n+k} }{ (n+k)! }
\lesssim \frac 1 { \sqrt{n+k} } \bigg( \frac {eN}{n+k} \bigg)^{n+k}
\le \frac{1}{\sqrt{n}} \bigg( \frac {e (n-1) }{2} \bigg)^{n+k}.
\end{align}
Then, by extending the sum to run over all $k \ge 0$, we obtain
\begin{align}
\frac 1 z A(n,z)  &= \frac 1 {n^{n-2}} \sum_{k=0}^{n} C(n,n+k) z^{k}
\nl&
\lesssim \frac{1}{n^{n-2}\sqrt n} \sum_{k=0}^{n}
	\bigg( \frac {e n }{2} \bigg)^{n+k}
	\bigg( \frac 2 e \bigg)^n   \bigg( \frac{ t^2 n }{ 3 k }\bigg)^{k/2}
	z^k  		
   \le
    n^{3/2} \sum_{k=0}^{\infty} \frac 1 {k^{k/2}}
    \biggl( \frac{etn^{3/2}z}{2\sqrt{3}} \bigg)^k ,
\end{align}
which converges for all $z>0$.

\smallskip \noindent (i)
If $n^{3/2}z \le b$ then the series on the right-hand side is bounded by a constant $C_b$, as required.

\smallskip \noindent (ii)
We set $x =  \frac{etn^{3/2}z}{2\sqrt{3}}$ and use
the asymptotic formula \cite[Lemma~4.1(i)]{JC04}
\begin{equation} \label{eq:kk_asymp}
     \sum_{k=0}^{\infty} \frac{1}{k^{k/2}} x^k
     \sim
     (4\pi e^{-1})^{1/2} x e^{\frac{1}{2e}x^2}
     \qquad
     \text{as $x \to\infty$}
\end{equation}
to get a bound for large $x$.
For smaller $x\ge0$, we simply bound by a constant.
The desired result then follows by absorbing the prefactor of \eqref{eq:kk_asymp} and another factor of $n^{3/2}z = \const\,x$ into the exponential.
This completes the proof.
\end{proof}

\begin{lemma}[Dense connected graphs]
\label{lemma:dense}
Let $n\ge 3$  and
$z\le \frac{3}{en}$.  Then
$B(n,z)\lesssim  z^2$.
\end{lemma}

\begin{proof}
Let $\nu = \lfloor n/2 \rfloor \ge 1$.
The crude bound  \eqref{eq:C_crude} gives
\begin{align} \label{eq:B_pf}
    B(n,z)
    \le \frac 1 {n^{n-2}}
    \sum_{k = \nu}^{\infty} \frac{N^{n+k}}{(n+k)!} z^{k+1}
    & = \frac {z^{1+\nu}} {n^{n-2}}
    \frac{N^{n+\nu}}{(n+\nu)!}
    \sum_{k = \nu}^{\infty}
    (Nz)^{k-\nu}\frac{(n+\nu)!}{(n+k)!}
    \nnb & \le \frac {z^{1+\nu}} {n^{n-2}}
    \frac{N^{n+\nu}}{(n+\nu)!}
    \sum_{m=0}^{\infty}
    \biggl( \frac{Nz}{n+\nu} \bigg)^{m}
    ,
\end{align}
since $(n+k)! \ge (n+\nu)! (n+ \nu)^{k - \nu}$.
Note that by our hypothesis
\begin{align}
\frac { Nz } { n+ \nu }
\le \frac { \half n (n-1) z } { n + (\half n - \half)}
< \frac { n (n-1)}{3n-3} z
= \frac 1 3  {nz}
\le \frac 1 e ,
\end{align}
so the geometric series in \eqref{eq:B_pf} is bounded by a constant.
For the prefactor in \eqref{eq:B_pf},
since $1+\nu \ge 2$ and $z\le \frac{3}{en}$ by hypothesis,
we have
\begin{equation}
z^{1+\nu} = z^2 z^{ \nu - 1}
\le \frac e 3 z^2 \Big( \frac{3}{e}\Big)^{\nu } n^{1 - \nu} .
\end{equation}
Also, using $N = \half n(n-1)$ and Stirling's formula,
\begin{equation}
\frac{ N^{n+\nu} }{(n+\nu)!}
\lesssim \frac{ [\half n (n-1) ]^{n+\nu} }
	{ \sqrt n ( \frac{n+\nu}e )^{n+\nu} }
= \frac {n^{n+\nu} } {\sqrt n }
\Big( \frac{e}{2}\Big)^{n+\nu}
    \Big( \frac{n-1}{n+\nu} \Big)^{n+\nu}	.
\end{equation}
Since $\frac{n-1}{n+\nu} \le \frac 2 3$,
together we obtain
\begin{align}
\frac {z^{1+\nu}} {n^{n-2}}  \frac{N^{n+\nu}}{(n+\nu)!}
\lesssim   z^2    \Big( \frac{3}{e}\Big)^{\nu } n^{5/2}
\Big( \frac{e}{3}\Big)^{n+\nu}
= z^2  n^{5/2}
\Big( \frac{e}{3}\Big)^{n} .
\end{align}
It follows that
$B(n,z) \lesssim z^2
\sup_{n \ge 3} \{ n^{5/2} (e/3)^n \} ,
$
and the proof is complete since $e < 3$.
\end{proof}

\subsection{Proof for connected subgraphs}

\begin{proof}[Proof of Theorem~\ref{thm:profile} for connected subgraphs]
Fix $s\in \R$ and let $p = 1 + sV^{-1/2}$.
We assume $V$ is large enough so that $p \le 3$.
As discussed around \eqref{eq:Deltabd},
it suffices to prove
\begin{align} \label{eq:animal_goal}
    \lim_{V\to\infty} V^{-1/4}\Delta_V(1 + sV^{-1/2})
    &=
    0.
\end{align}
By the definition of $\Delta_V$ in \eqref{eq:chiDelta} and by \eqref{eq:SAB},
\begin{align} \label{eq:diff}
    \Delta_{V}(p)
&\le \sum_{n=3}^{V}
	\binom{V-1}{n-1} \Big( \frac{p}{eV} \Big)^{n-1}n^{n-1}
	\Big( A \big(n, \frac{p}{eV} \big)
		+ B\big(n, \frac{p}{eV}\big) \Big) .
\end{align}
We write the part of the upper bound \eqref{eq:diff} that contains $A,B$ as $\Delta_V \supk A, \Delta_V \supk B$ respectively.

We start with $\Delta_V \supk B$
and use Lemma~\ref{lemma:dense} to bound $B$.
Let $z = p/(eV)$.
Since $n \le V$ and $p \le 3$ (for large $V$), we have $nz \le p / e \le 3/e$,
so Lemma~\ref{lemma:dense} applies and gives $B(n,z) \lesssim V^{-2}$.
Then, by comparing to $\chi_V^t$ in \eqref{eq:chit}, we find that
\begin{equation} \label{eq:DeltaB}
\Delta_V \supk B(p) \lesssim V^{-2} \chi_V^t(p).
\end{equation}
For $\Delta_V \supk A$,
we claim that
if both $b \ge 1$ and $V$ are sufficiently large, then
\begin{align} \label{eq:DeltaA}
\Delta_V \supk A(p)
\le C_{b} V^{-1/4} \chi_V^t(p)
	+ C_s b^{-1/2} V^{1/4} .
\end{align}
Since we already know that $V^{-1/4} \chi_V^t$ converges,
\eqref{eq:DeltaA} implies that
\begin{align}
0 \le
\limsup_{V\to \infty} \frac{ \Delta_V}{V^{1/4}}
\le \limsup_{V\to \infty} \frac{ \Delta_V \supk A + \Delta_V \supk B }{V^{1/4}}
\le C_{s} b^{-1/2}
\end{align}
for all $b$ sufficiently large.
But $\Delta_V$ does not depend on $b$,
so by taking the limit $b\to \infty$, we obtain \eqref{eq:animal_goal},
as desired.

It remains to prove \eqref{eq:DeltaA}.
We divide the sum defining $\Delta_V\supk A$ into two parts
$\Delta_V \supk 1, \Delta_V \supk 2$,
which sum over $n$ in the intervals $[3, b \sqrt V  ]$,
$(b \sqrt V , V]$ respectively.

For $\Delta_V \supk 1$, we have $n \le bV^{1/2}$ so
$n^{3/2} z = n^{3/2}p/(eV) \le c_{b} V^{-1/4}$,
so we can apply Lemma~\ref{lemma:sparse}(i) to
obtain $A(n,z) \le C_b' n^{3/2} z \le C_b V^{-1/4}$.
With the formula for $\chi^t_V$ in \eqref{eq:chit}, this gives
\begin{equation} \label{eq:pf_A1}
    \Delta_V \supk 1
     \le  C_b V^{-1/4}\sum_{n=3}^{\lfloor b\sqrt V \rfloor}
	\binom{V-1}{n-1} \Big( \frac{p}{eV} \Big)^{n-1}n^{n-1}
    \le C_b V^{-1/4}\chi_V^t(p).
\end{equation}
This provides the first term on the right-hand side of \eqref{eq:DeltaA}.

For $\Delta_V \supk 2$, we use $z=p/(eV)$ and Lemma~\ref{lemma:sparse}(ii) to see that
\begin{equation}
    A(n,z) \le C_\eps
    \exp\Bigl\{ \big(\frac{1}{24}+\eps \big) e\inv t^2 p^2 n^{3} / V^2 \Big\},
\end{equation}
Since $t = \sqrt{3e}$,
we have $\frac 1 {24} e\inv t^2 = \frac1 8$.
By choosing $\eps$ small,
and by using
$p = 1 + s V^{-1/2} \to 1$ as $V\to \infty$,
for $V$ sufficiently large (depending only on $\eps,s$) we have
\begin{equation}
A(n,z)
\le C_\eps \exp \Bigl\{ \frac 1 5 \frac { n^3 }{V^2 }	\Bigr\} .
\end{equation}
For these values of $V$, we thus have
\begin{equation}
\Delta_V \supk 2(1 + s V^{-1/2})
\lesssim \sum_{n=\lceil b\sqrt V \rceil}^{V}
	\binom{V-1}{n-1} \bigg( \frac{1 + s V^{-1/2}}{eV} \bigg)^{n-1}
	n^{n-1}
	\exp \Bigl\{ \frac 1 5 \frac { n^3 }{V^2 }	\Bigr\}  .
\end{equation}
We now follow
the argument used for $\chi_V\supk 3$ of trees in the paragraph containing \eqref{eq:pf_tree_3}.
Using $\frac { n^3 }{V^2 } \le \frac { n^2 }{V }$
and Lemma~\ref{lem:b} with $\gamma = \half$ and $\kappa = \half - \frac 15 > 0$,
we find that if $b$ is sufficiently large then
\begin{equation}
\label{eq:Delta2}
\Delta_V \supk 2(1 + s V^{-1/2}) \lesssim
    \sum_{n=\lceil b \sqrt V \rceil}^{V}
    e^{- (\frac 1 2 - \frac 1 5) n^2 / V}
    \frac{1}{\sqrt{ n}}
    e^{|s|n/\sqrt{V}}
\le C_{\abs s} b^{-1/2} V^{1/4}.
\end{equation}
This gives the second term on the right-hand side of \eqref{eq:DeltaA}
and concludes the proof.
\end{proof}

\begin{proof}[Proof of Theorem~\ref{thm:G0} for connected subgraphs]
As noted at \eqref{eq:Delta0bd}, it suffices to prove that
$\Delta_{V,0}(1+sV^{-1/2}) \to 0$ for all $s \ge 0$.
We write $p = 1 + sV^{-1/2}$ and follow the proof of Theorem~\ref{thm:profile}.
Compared to $\Delta_V$ for the susceptibility,
there is one less factor $n$ in $\Delta_{V,0}$,
so instead of \eqref{eq:diff} we now have
\begin{equation}
    \Delta_{V,0}(p)
\le \sum_{n=3}^{V}
	\binom{V-1}{n-1} \Big( \frac{p}{eV} \Big)^{n-1}n^{n-2}
	\Big( A \big(n, \frac{p}{eV} \big)
		+ B\big(n, \frac{p}{eV}\big) \Big) .
\end{equation}
As in \eqref{eq:DeltaB}, the contribution from $B$ obeys
\begin{equation}
    \Delta_{V,0}^{(B)}(p)
    \lesssim
    V^{-2}G_{V,0}^t(p) \lesssim V^{-2},
\end{equation}
so it vanishes in the limit.
For $\Delta_{V,0}^{(1)}(p)$, the same bound on $A$ that was used in \eqref{eq:pf_A1}
now gives $\Delta_{V,0}^{(1)}(p) \le C_b V^{-1/4} G_{V,0}(p)$.
For $\Delta_{V,0}^{(2)}(p)$, in \eqref{eq:Delta2} we now have an extra factor
$n$ in the denominator, so
Lemma~\ref{lem:b} with $\gamma = \frac 3 2$ gives
\begin{equation}
\Delta_{V,0} \supk 2(1 + s V^{-1/2}) \lesssim
    \sum_{n= \lceil b \sqrt V \rceil}^{V}
    e^{- (\frac 1 2 - \frac 1 5) n^2 / V}
    \frac{1}{n^{3/2}}
    e^{|s|n/\sqrt{V}}
\le C_{\abs s} b^{-3/2} V^{-1/4}.
\end{equation}
Altogether,
we have $\Delta_{V,0}(p) \lesssim V^{-1/4} \to 0$,
and the proof is complete.
\end{proof}

\section*{Acknowledgements}
The work of both authors was supported in part by NSERC of Canada.

\end{document}